\documentclass{amsart}
\usepackage{amssymb}
\usepackage{mathrsfs}
\usepackage{pifont}
\usepackage{amsfonts}
\usepackage{stmaryrd}
\usepackage{latexsym,amssymb,amsmath}

\newtheorem{theorem}{Theorem}[section]
\newtheorem{lemma}[theorem]{Lemma}

\theoremstyle{definition}

\newtheorem{proposition}[theorem]{Proposition}

\theoremstyle{remark}
\newtheorem{remark}[theorem]{Remark}

\numberwithin{equation}{section}

\begin{document}

\title{Abelian ideals with given dimension in Borel subalgebras}

\author{Li Luo}
\address{Institute of mathematics, Academy of
Mathematics and Systems Science, Chinese Academy of Sciences,
Beijing 100190, China}
\email{luoli@amss.ac.cn}






\begin{abstract}
A well-known Peterson's theorem says that the number of abelian
ideals in a Borel subalgebra of a rank-$r$ finite dimensional simple
Lie algebra is exactly $2^r$. In this paper, we determine the
dimensional distribution of abelian ideals in a Borel subalgebra of
finite dimensional simple Lie algebras, which is a refinement of the
Peterson's theorem capturing more Lie algebra invariants.
\end{abstract}

\maketitle


.

\section{Introduction}
Let $\mathfrak{g}$ be a finite dimensional complex simple Lie
algebra and let $\mathfrak{b}$ be a Borel subalgebra of
$\mathfrak{g}$. The study of abelian ideals in $\mathfrak{b}$ can be
traced back to the work by Schur \cite{Si} (1905). It has recently
drawn considerable attention. Kostant \cite{K1} (1998) mentioned
Peterson's $2^r$-theorem saying that the number of abelian ideals in
$\mathfrak{b}$ is exactly $2^r$, where $r$ is the rank of
$\mathfrak{g}$. Moreover, Kostant found a relation between abelian
ideals of a Borel subalgebra and the discrete series representations
of the Lie group. Spherical orbits were described by Panyushev and
R\"{o}hrle \cite{PR} (2001) in terms of abelian ideals. Furthermore,
Panyushev \cite{P1} (2003) discovered a correspondence of maximal
abelian ideals of a Borel subalgebra to long positive roots. Suter
\cite{S} (2004) determined the maximal dimension among abelian
subalgebras of a finite-dimensional simple Lie algebra purely in
terms of certain invariants and gave a uniform explanation for
Panyushev's result. Kostant \cite{K2} (2004) showed that the powers
of the Euler product and abelian ideals of a Borel subalgebra are
intimately related. Cellini and Papi \cite{CP3} (2004) had a
detailed study of certain remarkable posets which form a natural
partition of all abelian ideals of a Borel subalgebra. For
symplectic Lie algebras, we constructed a relationship between the
abelian ideals of Borel subalgebras and the cohomology of their
maximal nilpotent subalgebras \cite{L}. There are also many papers
on ad-nilpotent ideals in Borel subalgebras, such as \cite{CP1, CP2,
P2, Rg, R, Se}. It is also remarkable that the abelian ideals of
some tree diagram Lie algebras, which are analogues of Borel
subalgebras, are used to solve certain evolution partial
differential equations by Xu \cite{X} (2006).

Influenced by the above works, we will determine in this paper the
number of abelian ideals with given dimension in a Borel subalgebra.
It is a refinement of the Peterson's theorem capturing more Lie
algebra invariants.

There is an injection from abelian ideals to affine Weyl group, and
the dimension of an abelian ideal is equal to the length of the
corresponding element in affine Weyl group (c.f. \cite{S}). So our
results may be an accessorial tool to catch on the affine Weyl
group. While our results are much more precise than Peterson's
theorem, the consequences would be of broad interest. To compare the
results of type $B$ and $C$ (see Sec.4.1 and Sec.4.2) enable us to
be aware of that the distribution of abelian ideals does not depend
on the Weyl groups, though these groups played important roles in
Peterson's approach. On the other hand, the distributions of abelian
ideals for type $B$ and $D$ are uniform (see Sec.4.2 and Sec.4.3),
Although their root systems are different from each other. We can
expect this result because both of them correspond to orthogonal Lie
algebras. Moreover, our results can also check some data
corresponding to maximal abelian ideals, which are listed in
\cite{PR,Rg}.

The paper is organized as follows. In Section 2, we recall some
notations and definitions on root systems and give some lemmas which
are used in our calculation. Section 3 is devoted to the computation
of type $A$, and the result is described by a restricted partition
function. The computation of types $B$, $C$ and $D$ is in Section 4,
where the generating functions of the number of abelian ideals with
given dimension in Borel subalgebras is obtained. For exceptional
types $E$, $F$ and $G$, we enumerate their abelian ideals and obtain
the detailed information in Section 5.

\section{Preliminary}
Let $\mathfrak{g}$ be a finite dimension complex simple algebra with
rank $r$. Fix a Cartan subalgebra $\mathfrak{h}$ of $\mathfrak{g}$.
Denote by $\Phi$ its root system with a simple roots basis $\Pi$ and
positive roots $\Phi_+$. We have the Cartan root space decomposition
$\mathfrak{g}=\mathfrak{h}\bigoplus(\bigoplus_{\alpha\in
\Phi}\mathfrak{g}_\alpha)$. Then
$\mathfrak{b}=\mathfrak{h}\bigoplus(\bigoplus_{\alpha\in
\Phi_+}\mathfrak{g}_\alpha)$  is the associated Borel subalgebra.

Denote by $\mathcal {I}$ the set of all abelian ideals in
$\mathfrak{b}$. As mentioned in \cite{S}, there is a bijection as
follows:
\begin{eqnarray}
\Upsilon:=\{\Psi\subset \Phi_+\mid \Psi\dotplus \Phi_+\subset \Psi,
\Psi\dotplus \Psi=\emptyset\}&\leftrightarrow& \mathcal
{I}\nonumber\\\Psi&\mapsto&
I_\Psi:=\bigoplus_{\alpha\in\Psi}\mathfrak{g}_\alpha,
\end{eqnarray}
where $\Psi\dotplus \Phi_+:=(\Psi+\Phi_+)\cap\Phi_+$ and
$\Psi\dotplus \Psi:=(\Psi+\Psi)\cap\Phi_+$. It is clear that for any
$\Psi\in\Upsilon$,
\begin{equation}
\dim I_\Psi=|\Psi|.
\end{equation}
(In this paper, we always use $|A|$ for a set $A$ to denote the
total number of elements in $A$.)

Define
\begin{equation}
\mathcal {I}^{(i)}:=\{I\in\mathcal {I}\mid \dim I=i\}
\end{equation}
and
\begin{equation}
\Upsilon^{(i)}:=\{\Psi\in\Upsilon\mid |\Psi|=i\}.
\end{equation}
Thanks to (2.2), to enumerate $\mathcal {I}^{(i)}$ is as the same as
to enumerate $\Upsilon^{(i)}$.

Now we recall some notations and definitions (e.g. c.f.\cite{FSS,
H}), which will be used hereinafter.

The height of an element $\alpha=\sum_{\tau\in\Pi}k_\tau \tau$ is
defined by
\begin{equation}
\mbox{ht}(\alpha):=\sum_{\tau\in\Pi}k_\tau.
\end{equation}
There is a unique highest root in $\Phi$, which is denoted by
$\theta$.

There is a positive definite symmetric bilinear form $(\cdot\
,\cdot)$ on the root lattice
$\Lambda_r:=\mbox{Span}_{\mathbb{Z}}\Pi$ associated with the Killing
form on $\mathfrak{g}$. It is well known that for any two
nonproportional roots $\alpha, \beta$, if $(\alpha,\beta)>0$ (resp.
$<0$), then $\alpha-\beta$ (resp. $\alpha+\beta$) is a root.

We can define a partial ordering ``$\succ$'' on the root lattice
$\Lambda_r$ by
\begin{equation}
\alpha\succ\beta\quad\mbox{iff}\quad \mbox{$\alpha-\beta$ is a sum
of simple roots or $\alpha=\beta$}.
\end{equation}

A subset $\Psi\subset\Phi_+$ is called an increasing subset if for
any $\alpha,\beta\in\Phi_+$, the conditions $\alpha\in \Psi$ and
$\beta\succ\alpha$ imply $\beta\in \Psi$.

The following lemma will show that each $\Psi\in\Upsilon$ has to be
an increasing subset.

\begin{lemma}
For any $\alpha,\beta\in \Phi_+$ with $\beta-\alpha\succ0$, there
exist $\tau_1,\tau_2,\ldots,\tau_k\in\Pi$ such that
$\alpha+\tau_1,\alpha+\tau_1+\tau_2,\ldots,\alpha+\tau_1+\cdots+\tau_k=\beta\in\Phi_+$.
\end{lemma}
\begin{proof}
We will prove it by induction on $\mbox{ht}(\beta-\alpha)$. If
$\mbox{ht}(\beta-\alpha)=1$, the statement holds trivially.

Suppose the lemma holds for $\mbox{ht}(\beta-\alpha)=k-1,(k>1)$. Now
assume $\beta-\alpha=\sum_{\tau\in\Pi}k_{\tau}\tau$ with
$\sum_{\tau\in\Pi}k_{\tau}=k$. Since
\begin{equation*}
(\beta-\alpha,\beta)=(\beta-\alpha,\beta-\alpha)+(\beta-\alpha,\alpha)>(\beta-\alpha,\alpha),
\end{equation*}
we have either $(\beta-\alpha,\beta)>0$ or
$(\beta-\alpha,\alpha)<0$. Hence there must be a $\tau\in\Pi$ with
$k_\tau>0$ such that either $(\tau,\beta)>0$ or $(\tau,\alpha)<0$.
Then either $\beta-\tau$ or $\alpha+\tau$ is a positive root. By
inductive assumption on $\alpha$ and $\beta-\tau$ or on
$\alpha+\tau$ and $\beta$, we get that the statement holds for
$\mbox{ht}(\beta-\alpha)=k$.
\end{proof}

\begin{lemma}
For any $\alpha,\beta\in\Phi_+$, if there is a $\gamma\in\Phi_+$
with $\alpha+\beta\prec\gamma$, then there exist
$\alpha',\beta'\in\Phi_+$ such that
$\alpha'\succ\alpha,\beta'\succ\beta$ and $\alpha'+\beta'\in\Phi_+$.
\end{lemma}
\begin{proof}
We will also use induction on the height of $\gamma-\alpha-\beta$ to
prove this lemma.

If $\mbox{ht}(\gamma-\alpha-\beta)=0$ (i.e. $\gamma=\alpha+\beta$),
then we can take $\alpha'=\alpha$ and $\beta'=\beta$.

Suppose that the lemma holds for
$\mbox{ht}(\gamma-\alpha-\beta)=k-1,(k>0)$. Now assume
$\gamma-\alpha-\beta=\sum_{\tau\in\Pi}k_{\tau}\tau$ with
$\sum_{\tau\in\Pi}k_{\tau}=k$. Since
\begin{eqnarray*}
(\gamma-\alpha-\beta,
\gamma)&=&(\gamma-\alpha-\beta,\alpha)+(\gamma-\alpha-\beta,\beta)+(\gamma-\alpha-\beta,\gamma-\alpha-\beta)\\
&>&(\gamma-\alpha-\beta,\alpha)+(\gamma-\alpha-\beta,\beta),
\end{eqnarray*}
we have $(\gamma-\alpha-\beta,\alpha)<0$ or
$(\gamma-\alpha-\beta,\beta)<0$ or $(\gamma-\alpha-\beta,\gamma)>0$.
Hence there must be a $\tau\in\Pi$ with $k_\tau>0$ such that
$(\tau,\alpha)<0$ or $(\tau,\beta)<0$ or $(\tau,\gamma)>0$. If
$(\tau,\alpha)<0$ (resp. $(\tau,\beta)<0$ and $(\tau,\gamma)>0$),
then $\alpha_0=\alpha+\tau$ (resp. $\beta_0=\beta+\tau$ and
$\gamma_0=\gamma-\tau$) is a positive root. Take $\alpha_0$ (resp.
$\beta_0$ and $\gamma_0$) instead of $\alpha$ (resp. $\beta$ and
$\gamma$) and use the inductive assumption, then we get that the
statement holds for $\mbox{ht}(\gamma-\alpha-\beta)=k-1$.
\end{proof}

The above two lemmas (specially, take $\theta$ to be $\gamma$ in
Lemma 2.2) show immediately that
\begin{proposition}Take any $\Psi\in\Upsilon$.
(1) If $\alpha\in\Psi$ and $\alpha\prec\alpha'\in\Phi_+$, then
$\alpha'\in\Psi$. (2) If $\alpha,\beta\in\Phi_+$ and
$\alpha+\beta\prec\theta$, then it is impossible that both $\alpha$
and $\beta$ are in $\Psi$. In particular, if $2\alpha\prec\theta$,
then $\alpha\not\in\Psi$. \hfill$\Box$
\end{proposition}

The above proposition can help us predigest the enumeration of
abelian ideals. In fact, now we have a much more intuitional
description of $\Upsilon$ than in (2.1). That is,
\begin{equation}
\Upsilon=\{\Psi\in \Phi_+\mid\mbox{$\Psi$ is an increasing subset of
$\Phi_+$ and for any $\alpha,\beta\in\Psi$,
$\alpha+\beta\not\prec\theta$}\}.
\end{equation}
Particularly, we need not to consider all roots but only to consider
the roots in a subset $\Omega\subset\Phi_+$, which is defined by
\begin{equation}\Omega:=
\{\alpha\in\Phi_+\mid 2\alpha\not\prec\theta\}.
\end{equation}

Sometimes the subscript ``$_r$'' is used for a notation to emphasize
the rank $r$. For example, we use $\Omega_r$, $\Upsilon_r$, etc.

Define the minimal roots $\Psi_{\min}$ of any $\Psi\subset\Phi_+$ as
follows: $\alpha\in\Psi$ belongs to $\Psi_{\min}$ if and only if for
all $\beta\in\Phi_+$ with $\beta\prec\alpha$ and $\beta\neq\alpha$,
then $\beta\not\in\Psi$. Since each $\Psi\in\Upsilon$ is an
increasing subset, it can be unique determined by $\Psi_{\min}$.
Note that the elements of $\Psi_{\min}$ are pairwise incomparable
elements of $\Phi_+$. Inversely, every set of pairwise incomparable
elements of $\Phi_+$ with the sum of any two elements being not less
than or equal to $\theta$ is a $\Psi_{\min}$ for a unique
$\Psi\in\Upsilon$. We call such a set an incomparable admissible
subset of $\Phi_+$. Now there comes a bijection induced from (2.1)
as follows:
\begin{eqnarray}
 \mathcal {I}\leftrightarrow \{\mbox{ incomparable admissible subsets of $\Phi_+$}\},\quad
 I_{\Psi}\mapsto\Psi_{\min}.
\end{eqnarray}
This correspondence will convenience us to enumerate the abelian
ideals.

\section{Classical Type $A_r$}
In this section, we will determine the number of abelian ideals with
dimension $m$ in a Borel subalgebra of special linear Lie algebras
$\mathfrak{sl}_{r+1}(r\geq 1)$. Precisely, the number of abelian
ideals with dimension $m$ can be expressed by a restricted partition
function $P_{i,j}(m)$ which will be defined in (3.7).

For type $A_r$, we have
\begin{equation}
\Pi=\{\alpha_1=\epsilon_1-\epsilon_2,\alpha_2=\epsilon_2-\epsilon_3,\ldots,\alpha_r=\epsilon_r-\epsilon_{r+1}\}
\end{equation}
and
\begin{equation}
\Phi_+=\{\epsilon_i-\epsilon_j=\alpha_i+\alpha_{i+1}+\cdots+\alpha_{j-1}\mid
0\leq i<j\leq r+1\}.
\end{equation}

The highest root is
\begin{equation}
\theta=\epsilon_1-\epsilon_{r+1}=\alpha_1+\alpha_2+\cdots+\alpha_r.
\end{equation}

So any incomparable admissible subset has to be with the form
\begin{eqnarray}
&&\{\alpha_{i_1}+\alpha_{i_1+1}+\cdots+\alpha_{j_1},
\alpha_{i_2}+\alpha_{i_2+1}+\cdots+\alpha_{j_2},\ldots,
\alpha_{i_k}+\alpha_{i_k+1}+\cdots+\alpha_{j_k}\\
&&\quad\quad\quad\quad\quad\quad\quad\quad\mid
k\in \mathbb{Z}_{\geq0}, 1\leq i_1<i_2<\cdots<i_k\leq
j_1<j_2<\cdots<j_k\leq r\}.\nonumber
\end{eqnarray}

We shall compute the dimension of the abelian ideal determined by
$\Psi_{\min}$ with form (3.4). In fact, the set $\Psi\in\Upsilon$
determined by $\Psi_{\min}$ is
\begin{eqnarray}
\Psi=\{\alpha_1+\cdots+\alpha_{j_1},\alpha_2+\cdots+\alpha_{j_1},\ldots,\alpha_{i_1}+\cdots+\alpha_{j_1},\\
\alpha_1+\cdots+\alpha_{j_1+1},\alpha_2+\cdots+\alpha_{j_1+1},\ldots,\alpha_{i_1}+\cdots+\alpha_{j_1+1},\nonumber\\
\cdots\cdots\cdots\quad\quad\quad\quad\quad\quad\quad\quad\quad\quad\nonumber\\
\alpha_1+\cdots+\alpha_{j_2-1},\alpha_2+\cdots+\alpha_{j_2-1},\ldots,\alpha_{i_1}+\cdots+\alpha_{j_2-1},\nonumber\\
\alpha_1+\cdots+\alpha_{j_2},\alpha_2+\cdots+\alpha_{j_2},\ldots,\alpha_{i_2}+\cdots+\alpha_{j_2},\nonumber\\
\cdots\cdots\cdots\quad\quad\quad\quad\quad\quad\quad\quad\quad\quad\nonumber\\
\alpha_1+\cdots+\alpha_{j_3-1},\alpha_2+\cdots+\alpha_{j_3-1},\ldots,\alpha_{i_2}+\cdots+\alpha_{j_3-1},\nonumber\\
\cdots\cdots\cdots\cdots\cdots\cdots\cdots\quad\quad\quad\quad\quad\quad\quad\nonumber\\
\alpha_1+\cdots+\alpha_{j_k},\alpha_2+\cdots+\alpha_{j_k},\ldots,\alpha_{i_k}+\cdots+\alpha_{j_k},\nonumber\\
\cdots\cdots\cdots\quad\quad\quad\quad\quad\quad\quad\quad\quad\quad\nonumber\\
\alpha_1+\cdots+\alpha_{r},\alpha_2+\cdots+\alpha_{r},\ldots,\alpha_{i_k}+\cdots+\alpha_{r}\}.\nonumber
\end{eqnarray}
Hence
\begin{equation}
\dim I_{\Psi}=|\Psi|=\overbrace{(i_1+\cdots+i_1)}^{j_2-j_1}
+\overbrace{(i_2+\cdots+i_2)}^{j_3-j_2}
+\cdots+\overbrace{(i_k+\cdots+i_k)}^{r+1-j_k}.
\end{equation}
Therefore the number of abelian ideals with dimension $m$ is equal
to the number of different ways to give a partition
$(\lambda_1\geq\lambda_2\geq\cdots\geq\lambda_t>0)\vdash m$ with
$t+\lambda_1\leq r+1$.

Given $i,j\in \mathbb{Z}_{>0}$, define a restricted partition
function
\begin{equation}
P_{i,j}(m):=\sharp\mbox{ of
}\{(\lambda_1\geq\lambda_2\geq\cdots\geq\lambda_i>0)\vdash m\mid
\lambda_1\leq j\}, \quad(m\in\mathbb{Z}_{>0}).
\end{equation}
We have
\begin{proposition}
The number of abelian ideals with dimension $m\not=0$ in a Borel
subalgebra of $\mathfrak{sl}_{r+1}$ is equal to
$\sum_{i=1}^{r}P_{i,r+1-i}(m)$. Besides, there is a trivial abelian
ideal $\emptyset$, whose dimension is $0$. \hfill$\Box$
\end{proposition}
\begin{remark}
The restricted partition function $P_{i,j}(m)$ can be expressed by a
common combinatorial technique. It is equal to the coefficient of
$s^it^m$ in
\begin{eqnarray}
&&(1+st+(st)^2+\cdots)(1+st^2+(st^2)^2+\cdots)\cdots(1+st^j+(st^j)^2+\cdots)\\
&&=\frac{1}{(1-st)(1-st^2)\cdots(1-st^j)}.\nonumber
\end{eqnarray}
\end{remark}
\begin{remark} To take $(i_1,i_2,\ldots,i_k; j_1,j_2,\ldots,j_k)$ in (3.4) is as
the same as to take subsets in $\{1,2,\ldots,r\}$. Hence by (2.9),
we get Peterson's $2^r$ theorem for type $A$. Now Peterson's $2^r$
theorem and Proposition 3.1 imply a combinatorial identity:
\begin{equation}
1+\sum_{m=1}^\infty\sum_{i=1}^{r}P_{i,r+1-i}(m)=2^r.
\end{equation}
\end{remark}

\section{Classical Types $B_r$, $C_r$ and $D r$}
In this section, we will determine the number of abelian ideals with
given dimension in a Borel subalgebra of symplectic Lie algebras
$\mathfrak{sp}_{2r}(r\geq2)$ and orthogonal Lie algebras
$\mathfrak{o}_{2r}(r\geq3)$, $\mathfrak{o}_{2r+1}(r\geq4)$.
Precisely, the number of abelian ideals with given dimension can be
expressed as the coefficient of a generating polynomial.

The simplest case among type $B$, $C$ and $D$ is the second one. So
we will consider type $C$ first.

\subsection{Type $C_r$}
\

The result in this subsection is a corollary in \cite{L}, where a
connection between the abelian ideals in a Borel subalgebra and the
cohomology of a maximal nilpotent subalgebra was constructed. Here
we give a much more direct method to obtain the result again for
self-containment.

For type $C_r$, we have
\begin{equation}
\Pi=\{\alpha_1=\epsilon_1-\epsilon_2,\alpha_2=\epsilon_2-\epsilon_3,\ldots,
\alpha_{r-1}=\epsilon_{r-1}-\epsilon_{r},\alpha_r=2\epsilon_r\}
\end{equation}
and
\begin{eqnarray}
&&\Phi_+=\{\epsilon_i-\epsilon_j=\alpha_i+\cdots+\alpha_{j-1},
2\epsilon_i=2\alpha_i+\cdots+2\alpha_{r-1}+\alpha_r,2\epsilon_r=\alpha_r,\\
&&\quad\quad\quad
\epsilon_i+\epsilon_j=\alpha_i+\cdots+\alpha_{j-1}+2\alpha_j+\cdots+2\alpha_{r-1}+\alpha_r
\mid 1\leq i<j\leq r \}.\nonumber
\end{eqnarray}
The highest root is
\begin{equation}
\theta=2\epsilon_1=2\alpha_1+\cdots+2\alpha_{r-1}+\alpha_r.
\end{equation}

Recall the definition of $\Omega$ in (2.8). It is obvious that for
type $C$,
\begin{equation}
\Omega=\{\epsilon_i+\epsilon_j\mid1\leq i\leq j\leq r\}.
\end{equation}
Observe that
\begin{equation}
\epsilon_{i_1}+\epsilon_{j_1}\prec\epsilon_{i_1}+\epsilon_{j_1}
\quad\mbox{if and only if}\quad i_1\leq i_2 \mbox{ and } j_1\leq
j_2,
\end{equation}
where $i_1\leq j_1$ and $i_2\leq j_2$.

Moreover, for any $\alpha,\beta\in\Omega$, we have
$\alpha+\beta\not\prec\theta$. So by (2.7), we have
\begin{equation}
\Upsilon=\{\mbox{ increasing subset of }\Omega\}.
\end{equation}

Recall the subscript ``$_r$'' (especially $\Omega_r$ and
$\Upsilon_r$) explained below (2.8). Following is the Hasse diagram
of
$\Omega_r$ under the partial ordering ``$\prec$''.\\
(Note: In this paper, we always draw the maximal elements in the
bottom of the Hasse diagram for convenience. It is opposite to some
other literatures, where the maximal elements are drawn in the top
of a Hasse diagram.)

\setlength{\unitlength}{1pt}
\begin{picture}(0,0)
\put(5,-5){$2\epsilon_r$} \put(12,-8){\line(2,-1){10}}
\put(15,-18){$\epsilon_{r-1}+\epsilon_r$}
\put(12,-26){\line(2,1){10}} \put(49,-22){\line(2,-1){10}}
\put(5,-32){$2\epsilon_{r-1}$}
\put(44,-32){$\epsilon_{r-2}+\epsilon_{r}$}
\put(12,-35){\line(2,-1){10}} \put(49,-39){\line(2,1){10}}
\put(80,-35){\line(2,-1){10}}
\put(15,-45){$\epsilon_{r-2}+\epsilon_{r-1}$}
\put(78,-45){$\epsilon_{r-3}+\epsilon_{r}$}
\put(12,-53){\line(2,1){10}} \put(49,-49){\line(2,-1){10}}
\put(80,-53){\line(2,1){10}} \put(110,-49){\line(2,-1){10}}
\put(5,-59){$2\epsilon_{r-2}$}
\put(44,-59){$\epsilon_{r-3}+\epsilon_{r-1}$}
\put(107,-59){$\epsilon_{r-4}+\epsilon_{r}$}
\put(12,-62){\line(2,-1){10}} \put(49,-66){\line(2,1){10}}
\put(80,-62){\line(2,-1){10}} \put(110,-66){\line(2,1){10}}
\put(140,-62){\line(2,-1){10}}
\put(15,-72){$\epsilon_{r-3}+\epsilon_{r-2}$}
\put(78,-72){$\epsilon_{r-4}+\epsilon_{r-1}$}
\put(140,-72){$\epsilon_{r-5}+\epsilon_{r}$}
\put(12,-80){\line(2,1){10}} \put(49,-76){\line(2,-1){10}}
\put(80,-80){\line(2,1){10}} \put(110,-76){\line(2,-1){10}}
\put(140,-80){\line(2,1){10}} \put(170,-76){\line(2,-1){10}}
\put(5,-86){$\cdots$} \put(30,-86){$\cdots$} \put(66,-86){$\cdots$}
\put(125,-86){$\cdots$} \put(185,-86){$\cdots$}
\put(12,-88){\line(2,-1){10}} \put(49,-92){\line(2,1){10}}
\put(80,-88){\line(2,-1){10}} \put(110,-92){\line(2,1){10}}
\put(140,-88){\line(2,-1){10}} \put(170,-92){\line(2,1){10}}
\put(198,-88){\line(2,-1){10}} \put(5,-86){$\cdots$}
\put(30,-98){$\cdots$} \put(66,-98){$\cdots$} \put(95,-98){$\cdots$}
\put(125,-98){$\cdots$} \put(140,-98){$\epsilon_{2}+\epsilon_{r-1}$}
\put(200,-98){$\epsilon_{1}+\epsilon_{r}$}
\put(12,-106){\line(2,1){10}} \put(49,-102){\line(2,-1){10}}
\put(80,-106){\line(2,1){10}} \put(110,-102){\line(2,-1){10}}
\put(140,-106){\line(2,1){10}} \put(170,-102){\line(2,-1){10}}
\put(200,-106){\line(2,1){10}} \put(5,-86){$\cdots$}
\put(5,-112){$\cdots$} \put(30,-112){$\cdots$}
\put(66,-112){$\cdots$} \put(125,-112){$\cdots$}
\put(185,-112){$\cdots$} \put(12,-116){\line(2,-1){10}}
\put(49,-120){\line(2,1){10}} \put(80,-116){\line(2,-1){10}}
\put(110,-120){\line(2,1){10}} \put(140,-116){\line(2,-1){10}}
\put(170,-120){\line(2,1){10}}
\put(20,-126){$\epsilon_{3}+\epsilon_{4}$}
\put(85,-126){$\epsilon_{2}+\epsilon_{5}$}
\put(150,-126){$\epsilon_{1}+\epsilon_{6}$}
\put(12,-134){\line(2,1){10}} \put(49,-130){\line(2,-1){10}}
\put(80,-134){\line(2,1){10}} \put(110,-130){\line(2,-1){10}}
\put(140,-134){\line(2,1){10}} \put(5,-140){$2\epsilon_{3}$}
\put(54,-140){$\epsilon_{2}+\epsilon_{4}$}
\put(117,-140){$\epsilon_{1}+\epsilon_{5}$}
\put(12,-143){\line(2,-1){10}} \put(49,-148){\line(2,1){10}}
\put(80,-143){\line(2,-1){10}} \put(110,-148){\line(2,1){10}}
\put(20,-154){$\epsilon_{2}+\epsilon_{3}$}
\put(85,-154){$\epsilon_{1}+\epsilon_{4}$}
\put(12,-162){\line(2,1){10}} \put(49,-158){\line(2,-1){10}}
\put(80,-162){\line(2,1){10}} \put(5,-168){$2\epsilon_{2}$}
\put(54,-168){$\epsilon_{1}+\epsilon_{3}$}
\put(12,-171){\line(2,-1){10}} \put(49,-176){\line(2,1){10}}
\put(20,-182){$\epsilon_{1}+\epsilon_{2}$}
\put(12,-190){\line(2,1){10}} \put(5,-196){$2\epsilon_{1}$}
\put(90,-196){(Figure 1)}
\end{picture}

\newpage
Set
\begin{equation}
\Upsilon_r(n):= \left\{
\begin{array}{ll}
\{\emptyset\} & (n=0) \\
\{\Psi\in \Upsilon_r \mid \epsilon_1+\epsilon_n\in \Psi, \mbox{ but
} \epsilon_1+\epsilon_{n+1}\not\in \Psi\}&
(n=1,2,\ldots,r-1)\\
\{\Psi\in \Upsilon_r \mid \epsilon_1+\epsilon_r\in \Psi\} & (n=r)
\end{array}
\right.
\end{equation}

It is obvious that
\begin{equation}
\Upsilon_r=\bigcup_{n=0}^r\Upsilon_r(n), \quad\mbox{which is a
disjoint union.}
\end{equation}

Define
\begin{equation}
f_{C_r}(t):=\sum_{i=0}^\infty |\mathcal
{I}^{(i)}|t^i=\sum_{i=0}^\infty |\Upsilon_r^{(i)}|t^i.
\end{equation}

Although we restrict that $r\geq2$ in this subsection, we can also
set
\begin{equation}
\Omega_1=\{2\epsilon_1\} \quad ; \quad
\Upsilon_1^{(i)}=\left\{\begin{array}{cl} \{\emptyset\}, & (i=0)\\
\{2\epsilon_1\}, & (i=1)\\
\emptyset, & (i\geq2).
\end{array}
\right.
\end{equation}
and
\begin{equation}
\Omega_0=\Omega_{-1}=\emptyset \quad;\quad
\Upsilon_{0}^{(i)}=\Upsilon_{-1}^{(i)}=\left\{\begin{array}{cl} \{\emptyset\}, & (i=0)\\
\emptyset, & (i\geq1)
\end{array}
\right.
\end{equation} Hence we can also define
\begin{equation}
f_{C_1}(t)=1+t \quad\mbox{and}\quad
f_{C_{0}}(t)=f_{C_{-1}}(t)\equiv1
\end{equation}
for convenience.

Take any $\Psi\in\Upsilon_r(n)$. Besides the $n$ elements
$\epsilon_1+\epsilon_n, \epsilon_1+\epsilon_{n-1}, \ldots,
2\epsilon_1$, the rest elements in $\Psi$ should constitute an
increasing subset of a partially ordered set which is isomorphic to
$\Omega_{n-1}$. Hence by (4.8) we have
\begin{equation}
f_{C_r}(t)=\sum_{n=0}^rt^nf_{C_{n-1}}(t).
\end{equation}
Furthermore,
\begin{equation}
f_{C_r}(t)=t^rf_{C_{r-1}}(t)+\sum_{n=0}^{r-1}t^nf_{C_{n-1}}(t)=(1+t^r)f_{C_{r-1}}(t).
\end{equation}

Thus by (4.13) and (4.15) we obtain that
\begin{equation}
f_{C_r}(t)=(1+t)(1+t^2)\cdots(1+t^r).
\end{equation}

To summarize:
\begin{proposition}
The number of abelian ideals with dimension $m$ in a Borel
subalgebra of $\mathfrak{sp}_{2r}$ is equal to the coefficient of
$t^m$ in the polynomial (4.15). \hfill$\Box$
\end{proposition}
\begin{remark}
The generating function $f_{C_r}(t)$ is the quotient of the
poincar$\acute{\mbox{e}}$ polynomial of Weyl group of type $C_r$
divided by the poincar$\acute{\mbox{e}}$ polynomial of Weyl group of
type $A_{r-1}$. It is not accidental. In fact, we have showed in
\cite{L} that the cohomology of the nilradical of $\mathfrak{b}$ (in
type C) can be expressed as a sum over $S_r \times \Upsilon_r$,
where $S_r$ is the $r$-th symmetric group (also the Weyl group of
type $A_{r-1}$).
\end{remark}
\begin{remark}
It is clear that $f_{C_r}(1)=2^r$. Hence we get the Peterson's $2^r$
theorem for type $C$.
\end{remark}
\subsection{Type $B_r$}\

For type $B_r$, we have
\begin{equation}
\Pi=\{\alpha_1=\epsilon_1-\epsilon_2,\alpha_2=\epsilon_2-\epsilon_3,\ldots,
\alpha_{r-1}=\epsilon_{r-1}-\epsilon_{r},\alpha_r=\epsilon_r\}
\end{equation} and
\begin{eqnarray}
&&\Phi_+=\{\epsilon_i-\epsilon_j=\alpha_i+\cdots+\alpha_{j-1},\
\epsilon_k=\alpha_k+\cdots+\alpha_r,
\\&&\quad\quad
\epsilon_i+\epsilon_j=\alpha_i+\cdots+\alpha_{j-1}+2\alpha_j+\cdots+2\alpha_r\mid
1\leq i<j\leq r, 1\leq k\leq r \}.\nonumber
\end{eqnarray}
The highest root is
\begin{equation}
\theta=\epsilon_1+\epsilon_2=\alpha_1+2\alpha_2+\cdots+2\alpha_r.
\end{equation}
Hence
\begin{equation}
\Omega=\{\epsilon_1,\epsilon_1-\epsilon_j, \epsilon_i+\epsilon_j\mid
1\leq i<j\leq r\}.
\end{equation}

However, by proposition 2.3, there are several restrictive
conditions for us to take incomparable admissible subset in
$\Omega$. They are induced from
\begin{equation}
(\epsilon_1-\epsilon_r)+(\epsilon_2+\epsilon_r)=(\epsilon_1-\epsilon_{r-1})+(\epsilon_2+\epsilon_{r-1})
=\cdots=(\epsilon_1-\epsilon_3)+(\epsilon_2+\epsilon_3)=\theta.
\end{equation}
(Hereinafter, a restrictive condition induced from
$\alpha+\beta=\theta$ means that any two elements
$\alpha',\beta'\in\Omega$ with $\alpha'\prec\alpha$ and
$\beta'\prec\beta$ can not be both in a $\Psi\in\Upsilon$. For
example, the restrictive condition, that $\epsilon_1-\epsilon_2$ and
$\epsilon_2+\epsilon_3$ can not be both in a $\Psi\in\Upsilon$, is
induced from
$(\epsilon_1-\epsilon_3)+(\epsilon_2+\epsilon_3)=\theta$ since
$\epsilon_1-\epsilon_2\prec\epsilon_1-\epsilon_3$ and
$\epsilon_2+\epsilon_3\prec\epsilon_2+\epsilon_3$.)

We draw the Hasse diagram of $\Omega$ as follows.

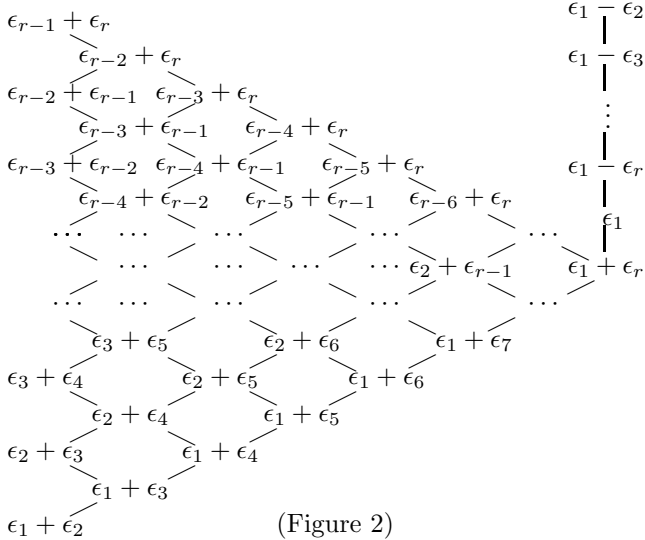
\begin{picture}(0,0)
\put(-12,-5){$\epsilon_{r-1}+\epsilon_r$}
\put(12,-8){\line(2,-1){10}}
\put(15,-18){$\epsilon_{r-2}+\epsilon_r$}
\put(12,-26){\line(2,1){10}} \put(49,-22){\line(2,-1){10}}
\put(-12,-32){$\epsilon_{r-2}+\epsilon_{r-1}$}
\put(44,-32){$\epsilon_{r-3}+\epsilon_{r}$}
\put(12,-35){\line(2,-1){10}} \put(49,-39){\line(2,1){10}}
\put(80,-35){\line(2,-1){10}}
\put(15,-45){$\epsilon_{r-3}+\epsilon_{r-1}$}
\put(78,-45){$\epsilon_{r-4}+\epsilon_{r}$}
\put(12,-53){\line(2,1){10}} \put(49,-49){\line(2,-1){10}}
\put(80,-53){\line(2,1){10}} \put(110,-49){\line(2,-1){10}}
\put(-12,-59){$\epsilon_{r-3}+\epsilon_{r-2}$}
\put(44,-59){$\epsilon_{r-4}+\epsilon_{r-1}$}
\put(107,-59){$\epsilon_{r-5}+\epsilon_{r}$}
\put(12,-62){\line(2,-1){10}} \put(49,-66){\line(2,1){10}}
\put(80,-62){\line(2,-1){10}} \put(110,-66){\line(2,1){10}}
\put(140,-62){\line(2,-1){10}}
\put(15,-72){$\epsilon_{r-4}+\epsilon_{r-2}$}
\put(78,-72){$\epsilon_{r-5}+\epsilon_{r-1}$}
\put(140,-72){$\epsilon_{r-6}+\epsilon_{r}$}
\put(12,-80){\line(2,1){10}} \put(49,-76){\line(2,-1){10}}
\put(80,-80){\line(2,1){10}} \put(110,-76){\line(2,-1){10}}
\put(140,-80){\line(2,1){10}} \put(170,-76){\line(2,-1){10}}
\put(5,-86){$\cdots$} \put(30,-86){$\cdots$} \put(66,-86){$\cdots$}
\put(125,-86){$\cdots$} \put(185,-86){$\cdots$}
\put(12,-88){\line(2,-1){10}} \put(49,-92){\line(2,1){10}}
\put(80,-88){\line(2,-1){10}} \put(110,-92){\line(2,1){10}}
\put(140,-88){\line(2,-1){10}} \put(170,-92){\line(2,1){10}}
\put(198,-88){\line(2,-1){10}} \put(5,-86){$\cdots$}
\put(30,-98){$\cdots$} \put(66,-98){$\cdots$} \put(95,-98){$\cdots$}
\put(125,-98){$\cdots$} \put(140,-98){$\epsilon_{2}+\epsilon_{r-1}$}
\put(200,-98){$\epsilon_{1}+\epsilon_{r}$}
\put(12,-106){\line(2,1){10}} \put(49,-102){\line(2,-1){10}}
\put(80,-106){\line(2,1){10}} \put(110,-102){\line(2,-1){10}}
\put(140,-106){\line(2,1){10}} \put(170,-102){\line(2,-1){10}}
\put(200,-106){\line(2,1){10}} \put(5,-86){$\cdots$}
\put(5,-112){$\cdots$} \put(30,-112){$\cdots$}
\put(66,-112){$\cdots$} \put(125,-112){$\cdots$}
\put(185,-112){$\cdots$} \put(12,-116){\line(2,-1){10}}
\put(49,-120){\line(2,1){10}} \put(80,-116){\line(2,-1){10}}
\put(110,-120){\line(2,1){10}} \put(140,-116){\line(2,-1){10}}
\put(170,-120){\line(2,1){10}}
\put(20,-126){$\epsilon_{3}+\epsilon_{5}$}
\put(85,-126){$\epsilon_{2}+\epsilon_{6}$}
\put(150,-126){$\epsilon_{1}+\epsilon_{7}$}
\put(12,-134){\line(2,1){10}} \put(49,-130){\line(2,-1){10}}
\put(80,-134){\line(2,1){10}} \put(110,-130){\line(2,-1){10}}
\put(140,-134){\line(2,1){10}}
\put(-12,-140){$\epsilon_{3}+\epsilon_4$}
\put(54,-140){$\epsilon_{2}+\epsilon_{5}$}
\put(117,-140){$\epsilon_{1}+\epsilon_{6}$}
\put(12,-143){\line(2,-1){10}} \put(49,-148){\line(2,1){10}}
\put(80,-143){\line(2,-1){10}} \put(110,-148){\line(2,1){10}}
\put(20,-154){$\epsilon_{2}+\epsilon_{4}$}
\put(85,-154){$\epsilon_{1}+\epsilon_{5}$}
\put(12,-162){\line(2,1){10}} \put(49,-158){\line(2,-1){10}}
\put(80,-162){\line(2,1){10}}
\put(-12,-168){$\epsilon_{2}+\epsilon_3$}
\put(54,-168){$\epsilon_{1}+\epsilon_{4}$}
\put(12,-171){\line(2,-1){10}} \put(49,-176){\line(2,1){10}}
\put(20,-182){$\epsilon_{1}+\epsilon_{3}$}
\put(12,-190){\line(2,1){10}}
\put(-12,-196){$\epsilon_{1}+\epsilon_2$}
\put(200,0){$\epsilon_1-\epsilon_2$} \put(214,-11){\line(0,1){10}}
\put(200,-18){$\epsilon_1-\epsilon_3$} \put(214,-29){\line(0,1){10}}
\put(214,-43){$\vdots$} \put(214,-55){\line(0,1){10}}
\put(200,-60){$\epsilon_1-\epsilon_r$} \put(214,-73){\line(0,1){10}}
\put(213,-80){$\epsilon_1$} \put(214,-90){\line(0,1){10}}
\put(90,-196){(Figure 2)}
\end{picture}
\vspace{7cm}

Set
\begin{equation}
\Upsilon_r(n):=\left\{ \begin{array}{ll} \{\Psi\in\Upsilon\mid
\epsilon_1-\epsilon_2\in \Psi\},& (n=2)\\
 \{\Psi\in\Upsilon\mid
\epsilon_1-\epsilon_n\in \Psi \mbox{ but
}\epsilon_1-\epsilon_{n-1}\not\in\Psi\},& (n=3,\ldots,r)\\
\{\Psi\in\Upsilon\mid \epsilon_1\in\Psi \mbox{ but
}\epsilon_1-\epsilon_{r}\not\in\Psi\}, &(n=r+1)\\
\{\Psi\in\Upsilon\mid\epsilon_1\not\in\Psi\}, &(n=r+2).
\end{array}\right.
\end{equation}
It is clear that
\begin{equation}
\Upsilon_r=\bigcup_{n=2}^{r+2}\Upsilon_r(n), \quad\mbox{which is a
disjoint union.}
\end{equation}

Define
\begin{equation} f_{B_r}(t):=\sum_{i=0}^\infty |\mathcal
{I}^{(i)}|t^i=\sum_{i=0}^\infty
|\Upsilon_r^{(i)}|t^i=\sum_{n=2}^{r+2}\sum_{i=0}^\infty
|\Upsilon_r^{(i)}(n)|t^i,
\end{equation}
where $\Upsilon_r^{(i)}(n):=\Upsilon_r^{(i)}\bigcap\Upsilon_r(n)$.

Observe that there is an isomorphism between two partially ordered
sets:
\begin{eqnarray}
\Omega_r\backslash\{\epsilon_1-\epsilon_i,\epsilon_1\mid 1\leq i\leq
r\}&\simeq&\Omega_{r-1}^C:\\
\epsilon_i+\epsilon_j(i<j)&\mapsto&\epsilon_i+\epsilon_{j-1}(i\leq
j-1),\nonumber
\end{eqnarray}
where $\Omega_{r-1}^C$ is the set $\Omega_{r-1}$ of type $C$. Hence
\begin{equation}
\sum_{i=0}^\infty
|\Upsilon_r^{(i)}(r+2)|t^i=f_{C_{r-1}}(t)=(1+t)(1+t^2)\cdots(1+t^{r-1}).
\end{equation}
Furthermore, for any $\Psi\in\Upsilon_r(n)\ (3\leq n\leq r+1)$,
besides the $2r-n+1$ elements:
$\epsilon_1+\epsilon_2,\epsilon_1+\epsilon_3,\ldots,\epsilon_1+\epsilon_{r},\epsilon_1,\epsilon_1-\epsilon_r,\epsilon_1-\epsilon_{r-1},\ldots,\epsilon_1-\epsilon_n$
(when $n=r+1$, $\epsilon_1-\epsilon_{r+1}$ means $\epsilon_1$ for
convenience), the rest elements in $\Psi$ should constitute an
increasing subset of a partially ordered set which is isomorphic to
$\Omega_{n-3}^C$. Hence for $3\leq n\leq r+1$,
\begin{equation}
\sum_{i=0}^\infty
|\Upsilon_r^{(i)}(n)|t^i=t^{2r-n+1}f_{C_{n-3}}(t)=t^{2r-n+1}(1+t)(1+t^2)\cdots(1+t^{n-3}).
\end{equation}

At last, we can get immediately that
\begin{equation}
\Upsilon_r(2)=\{\{\epsilon_1\pm\epsilon_i,\epsilon_1\mid 1\leq i\leq
r\}\}.
\end{equation}
Hence
\begin{equation}
\sum_{i=0}^\infty|\Upsilon_r^{(i)}(2)|t^i=t^{2r-1}.
\end{equation}

Combine (4.23), (4.25), (4.26) and (4.28), then we get
\begin{proposition}
The number of abelian ideals with dimension $m$ in a Borel
subalgebra of $\mathfrak{o}_{2r}$ is equal to the coefficient of
$t^m$ in the polynomial
\begin{equation}
 f_{B_r}(t)=t^{2r-1}+\sum_{k=0}^{r-2}t^{2r-k-2}\prod_{j=1}^{k}(1+t^j)+\prod_{j=1}^{r-1}(1+t^j).
\end{equation}\hfill$\Box$
\end{proposition}
\begin{remark}
Let $t=1$ in (4.29), then we get $f_{B_r}(1)=2^r$, which is
Peterson's $2^r$ theorem for type $B$.
\end{remark}

\subsection{Type $D_r$}\

For type $D_r$, we have
\begin{equation}
\Pi=\{\alpha_1=\epsilon_1-\epsilon_2,\alpha_2=\epsilon_2-\epsilon_3,\ldots,
\alpha_{r-1}=\epsilon_{r-1}-\epsilon_{r},\alpha_r=\epsilon_{r-1}+\epsilon_r\}
\end{equation} and
\begin{eqnarray}
&&\Phi_+= \{\epsilon_i-\epsilon_j=\alpha_i+\cdots+\alpha_{j-1}
(1\leq i<j\leq r),\\
&&\quad\quad\quad\epsilon_i+\epsilon_r=\alpha_i+\cdots+\alpha_{r-2}+\alpha_{r} (1\leq i<r),\nonumber\\
&&
\epsilon_i+\epsilon_j=\alpha_i+\cdots+\alpha_{j-1}+2\alpha_{j}+\cdots+2\alpha_{r-2}+\alpha_{r-1}+\alpha_r(1\leq
i<j<r)\}.\nonumber
\end{eqnarray}
The highest root is
\begin{equation}
\theta=\epsilon_1+\epsilon_2=\alpha_1+2\alpha_2+\cdots+2\alpha_{r-2}+\alpha_{r-1}+\alpha_r.
\end{equation}
Hence
\begin{equation}
\Omega=\{\epsilon_i+\epsilon_j,\epsilon_1-\epsilon_j,\epsilon_i-\epsilon_r\mid
1\leq i<j\leq r\}.
\end{equation}

We draw the Hasse diagram of $\Omega$ as follows.

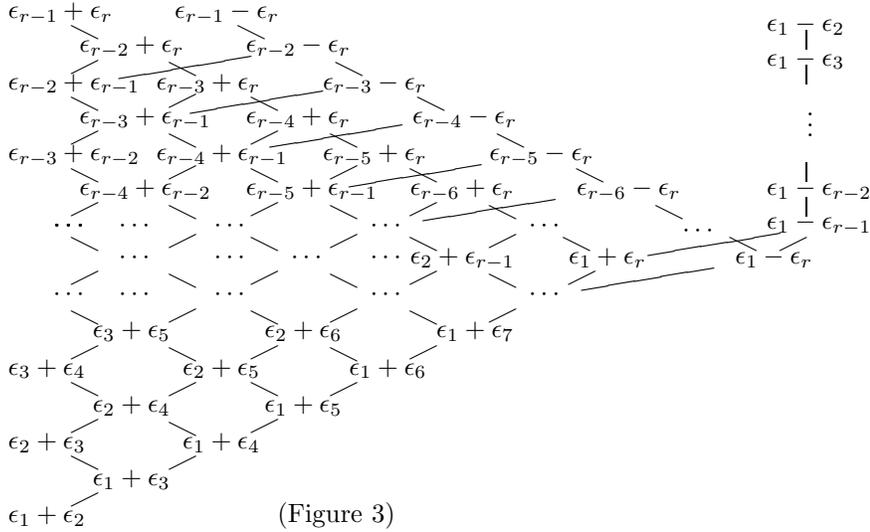
\begin{picture}(0,0)
\put(-12,-5){$\epsilon_{r-1}+\epsilon_r$}
\put(51,-5){$\epsilon_{r-1}-\epsilon_r$}
\put(12,-8){\line(2,-1){10}} \put(75,-8){\line(2,-1){10}}
\put(15,-18){$\epsilon_{r-2}+\epsilon_r$}
\put(78,-18){$\epsilon_{r-2}-\epsilon_r$}
\put(30,-28){\line(6,1){50}} \put(12,-26){\line(2,1){10}}
\put(49,-22){\line(2,-1){10}} \put(112,-22){\line(2,-1){10}}
\put(-12,-32){$\epsilon_{r-2}+\epsilon_{r-1}$}
\put(44,-32){$\epsilon_{r-3}+\epsilon_{r}$}
\put(107,-32){$\epsilon_{r-3}-\epsilon_{r}$}
\put(57,-41){\line(6,1){50}} \put(12,-35){\line(2,-1){10}}
\put(49,-39){\line(2,1){10}} \put(80,-35){\line(2,-1){10}}
\put(143,-35){\line(2,-1){10}}
\put(15,-45){$\epsilon_{r-3}+\epsilon_{r-1}$}
\put(78,-45){$\epsilon_{r-4}+\epsilon_{r}$}
\put(141,-45){$\epsilon_{r-4}-\epsilon_{r}$}
\put(87,-54){\line(6,1){50}} \put(12,-53){\line(2,1){10}}
\put(49,-49){\line(2,-1){10}} \put(80,-53){\line(2,1){10}}
\put(110,-49){\line(2,-1){10}} \put(173,-49){\line(2,-1){10}}
\put(-12,-59){$\epsilon_{r-3}+\epsilon_{r-2}$}
\put(44,-59){$\epsilon_{r-4}+\epsilon_{r-1}$}
\put(107,-59){$\epsilon_{r-5}+\epsilon_{r}$}
\put(170,-59){$\epsilon_{r-5}-\epsilon_{r}$}
\put(117,-69){\line(6,1){50}} \put(12,-62){\line(2,-1){10}}
\put(49,-66){\line(2,1){10}} \put(80,-62){\line(2,-1){10}}
\put(110,-66){\line(2,1){10}} \put(140,-62){\line(2,-1){10}}
\put(203,-62){\line(2,-1){10}}
\put(15,-72){$\epsilon_{r-4}+\epsilon_{r-2}$}
\put(78,-72){$\epsilon_{r-5}+\epsilon_{r-1}$}
\put(140,-72){$\epsilon_{r-6}+\epsilon_{r}$}
\put(203,-72){$\epsilon_{r-6}-\epsilon_{r}$}
\put(145,-82){\line(6,1){50}} \put(12,-80){\line(2,1){10}}
\put(49,-76){\line(2,-1){10}} \put(80,-80){\line(2,1){10}}
\put(110,-76){\line(2,-1){10}} \put(140,-80){\line(2,1){10}}
\put(170,-76){\line(2,-1){10}} \put(233,-76){\line(2,-1){10}}
\put(5,-86){$\cdots$} \put(30,-86){$\cdots$} \put(66,-86){$\cdots$}
\put(125,-86){$\cdots$} \put(185,-86){$\cdots$}
\put(243,-88){$\cdots$} \put(12,-88){\line(2,-1){10}}
\put(49,-92){\line(2,1){10}} \put(80,-88){\line(2,-1){10}}
\put(110,-92){\line(2,1){10}} \put(140,-88){\line(2,-1){10}}
\put(170,-92){\line(2,1){10}} \put(198,-88){\line(2,-1){10}}
\put(261,-88){\line(2,-1){10}} \put(5,-86){$\cdots$}
\put(30,-98){$\cdots$} \put(66,-98){$\cdots$} \put(95,-98){$\cdots$}
\put(125,-98){$\cdots$} \put(140,-98){$\epsilon_{2}+\epsilon_{r-1}$}
\put(200,-98){$\epsilon_{1}+\epsilon_{r}$}
\put(263,-98){$\epsilon_{1}-\epsilon_{r}$}
\put(205,-108){\line(6,1){50}} \put(12,-106){\line(2,1){10}}
\put(49,-102){\line(2,-1){10}} \put(80,-106){\line(2,1){10}}
\put(110,-102){\line(2,-1){10}} \put(140,-106){\line(2,1){10}}
\put(170,-102){\line(2,-1){10}} \put(200,-106){\line(2,1){10}}
\put(5,-86){$\cdots$} \put(5,-112){$\cdots$} \put(30,-112){$\cdots$}
\put(66,-112){$\cdots$} \put(125,-112){$\cdots$}
\put(185,-112){$\cdots$} \put(12,-116){\line(2,-1){10}}
\put(49,-120){\line(2,1){10}} \put(80,-116){\line(2,-1){10}}
\put(110,-120){\line(2,1){10}} \put(140,-116){\line(2,-1){10}}
\put(170,-120){\line(2,1){10}}
\put(20,-126){$\epsilon_{3}+\epsilon_{5}$}
\put(85,-126){$\epsilon_{2}+\epsilon_{6}$}
\put(150,-126){$\epsilon_{1}+\epsilon_{7}$}
\put(12,-134){\line(2,1){10}} \put(49,-130){\line(2,-1){10}}
\put(80,-134){\line(2,1){10}} \put(110,-130){\line(2,-1){10}}
\put(140,-134){\line(2,1){10}}
\put(-12,-140){$\epsilon_{3}+\epsilon_4$}
\put(54,-140){$\epsilon_{2}+\epsilon_{5}$}
\put(117,-140){$\epsilon_{1}+\epsilon_{6}$}
\put(12,-143){\line(2,-1){10}} \put(49,-148){\line(2,1){10}}
\put(80,-143){\line(2,-1){10}} \put(110,-148){\line(2,1){10}}
\put(20,-154){$\epsilon_{2}+\epsilon_{4}$}
\put(85,-154){$\epsilon_{1}+\epsilon_{5}$}
\put(12,-162){\line(2,1){10}} \put(49,-158){\line(2,-1){10}}
\put(80,-162){\line(2,1){10}}
\put(-12,-168){$\epsilon_{2}+\epsilon_3$}
\put(54,-168){$\epsilon_{1}+\epsilon_{4}$}
\put(12,-171){\line(2,-1){10}} \put(49,-176){\line(2,1){10}}
\put(20,-182){$\epsilon_{1}+\epsilon_{3}$}
\put(12,-190){\line(2,1){10}}
\put(-12,-196){$\epsilon_{1}+\epsilon_2$}
\put(275,-85){$\epsilon_1-\epsilon_{r-1}$}
\put(280,-93){\line(2,1){10}} \put(230,-95){\line(6,1){50}}
\put(290,-80){\line(0,1){7}}
\put(275,-72){$\epsilon_1-\epsilon_{r-2}$}
\put(290,-67){\line(0,1){7}} \put(290,-50){$\vdots$}
\put(290,-30){\line(0,1){7}}
\put(275,-23){$\epsilon_1-\epsilon_{3}$}
\put(290,-17){\line(0,1){7}}
\put(275,-10){$\epsilon_1-\epsilon_{2}$}
\put(90,-196){(Figure 3)}
\end{picture}
\vspace{7.5cm}

There are also several restrictive conditions induced by
\begin{equation}
(\epsilon_1-\epsilon_3)+(\epsilon_2+\epsilon_3)=(\epsilon_1-\epsilon_4)+(\epsilon_2+\epsilon_4)=\cdots=(\epsilon_1-\epsilon_r)+(\epsilon_2+\epsilon_r)=(\epsilon_1+\epsilon_r)+(\epsilon_2-\epsilon_r)=\theta.
\end{equation}

Set
\begin{equation}
\Upsilon_r(n):=\left\{
\begin{array}{ll}
\{\Psi\in\Upsilon\mid \epsilon_1-\epsilon_2\in\Psi\},& (n=2)\\
\{\Psi\in\Upsilon\mid \epsilon_1-\epsilon_n\in\Psi\mbox{ but
}\epsilon_1-\epsilon_{n-1}\not\in\Psi\},&(3\leq n\leq r-1)\\
\{\Psi\in\Upsilon\mid \epsilon_1\pm\epsilon_r\in \Psi \mbox{ but
}\epsilon_1-\epsilon_{r-1}\not\in\Psi\},& (n=r)\\
\{\Psi\in\Upsilon\mid \mbox{either }\epsilon_1+\epsilon_r\not\in\Psi
\mbox{ or }\epsilon_1-\epsilon_r\not\in\Psi\},&(n=r+1).
\end{array} \right.
\end{equation}
and
\begin{equation}
\Upsilon_r^{(i)}(n):=\Upsilon_r^{(i)}\bigcap\Upsilon_r(n).
\end{equation}
Also we have
\begin{equation}
\Upsilon_r=\bigcup_{n=2}^{r}\Upsilon_r(n), \quad\mbox{which is a
disjoint union.}
\end{equation}

Define
\begin{equation} f_{D_r}(t):=\sum_{i=0}^\infty |\mathcal
{I}^{(i)}|t^i=\sum_{i=0}^\infty
|\Upsilon_r^{(i)}|t^i=\sum_{n=2}^{r}\sum_{i=0}^\infty
|\Upsilon_r^{(i)}(n)|t^i,
\end{equation}

In a similar way to (4.26) and (4.28), we can obtain that for $3\leq
n\leq r$
\begin{equation}
\sum_{i=0}^\infty
|\Upsilon_r^{(i)}(n)|t^i=t^{2r-n}f_{C_{n-3}}(t)=t^{2r-n}(1+t)(1+t^2)\cdots(1+t^{n-3})
\end{equation}
 and
\begin{equation}
\sum_{i=0}^\infty|\Upsilon_r^{(i)}(2)|t^i=t^{2r-2}.
\end{equation}
Now we are going to calculate
$\sum_{i=0}^\infty|\Upsilon_r^{(i)}(r+1)|t^i$.

Set
\begin{equation}
\Upsilon_r(r+1)^+:=\{\Psi\in\Upsilon_{r}(r+1)\mid
\epsilon_1+\epsilon_r\in\Psi\},
\end{equation}
\begin{equation}
\Upsilon_r(r+1)^-:=\{\Psi\in\Upsilon_{r}(r+1)\mid
\epsilon_1-\epsilon_r\in\Psi\}
\end{equation}
and
\begin{equation} \Upsilon_r(r+1)^0:=\{\Psi\in\Upsilon_{r}(r+1)\mid
\epsilon_1\pm\epsilon_r\not\in\Psi\}.
\end{equation}

we have
\begin{equation}
\Upsilon_r(r+1)=\Upsilon_r(r+1)^+\bigcup
\Upsilon_r(r+1)^-\bigcup\Upsilon_r(r+1)^0,
\end{equation}
which is a disjoint union. Hence
\begin{equation}
\sum_{i=0}^\infty|\Upsilon_r^{(i)}(r+1)|t^i=\sum_{i=0}^\infty|\Upsilon_r^{(i)}(r+1)^+|t^i
+\sum_{i=0}^\infty|\Upsilon_r^{(i)}(r+1)^-|t^i+\sum_{i=0}^\infty|\Upsilon_r^{(i)}(r+1)^0|t^i,
\end{equation}
where
$\Upsilon_r^{(i)}(r+1)^{\pm,0}:=\Upsilon_r(r+1)^{\pm,0}\bigcap\Upsilon_r^{(i)}$.

Observe that there are several isomorphisms between some partially
ordered sets:
\begin{eqnarray}
\Omega_{r-1}^C&\simeq&\Omega_r\backslash\{\epsilon_1-\epsilon_n,\epsilon_l-\epsilon_r\mid
2\leq n\leq r-1, 1\leq l\leq
r-1\}\\\nonumber&\simeq&\Omega_r\backslash\{\epsilon_1-\epsilon_n,\epsilon_l+\epsilon_r\mid
2\leq n\leq r-1, 1\leq l\leq r-1\}
\end{eqnarray}
and
\begin{eqnarray}
\Omega_{r-2}^C&\simeq&\Omega_r\backslash\{\epsilon_1-\epsilon_n,\epsilon_l\pm\epsilon_r\mid
2\leq n\leq r-1, 1\leq l\leq r-1\}.
\end{eqnarray}
Therefore
\begin{eqnarray}
&&\sum_{i=0}^\infty|\Upsilon_r^{(i)}(r+1)^+|t^i
=\sum_{i=0}^\infty|\Upsilon_r^{(i)}(r+1)^-|t^i\\
\nonumber&&=f_{C_{r-1}}(t)-f_{C_{r-2}}(t)=t^{r-1}(1+t)(1+t^2)\cdots(1+t^{r-2})
\end{eqnarray}
and
\begin{equation}
\sum_{i=0}^\infty|\Upsilon_r^{(i)}(r+1)^0|t^i=f_{C_{r-2}}(t)=(1+t)(1+t^2)\cdots(1+t^{r-2}).
\end{equation}

Thanks to (4.38),(4.39),(4.40),(4.45),(4.48) and (4.49), we have
\begin{proposition}
The number of abelian ideals with dimension $m$ in a Borel
subalgebra of $\mathfrak{o}_{2r+1}$ is equal to the coefficient of
$t^m$ in the polynomial
\begin{equation}
 f_{D_r}(t)=t^{2r-2}+\sum_{k=0}^{r-2}t^{2r-k-3}\prod_{j=1}^{k}(1+t^j)+\prod_{j=1}^{r-1}(1+t^j).
\end{equation}\hfill$\Box$
\end{proposition}
\begin{remark}
Peterson's $2^r$ theorem for type $D$ can be obtained by taking
$t=1$ in (4.50). That is $f_{D_r}(1)=2^r$.
\end{remark}

\section{Exceptional Types $E_6$, $E_7$, $E_8$, $F_4$ and $G_2$}
In this section, we will deal with the exceptional types. Precisely,
we will give the Hasse diagrams of $\Omega$ of these types and then
list all data on the number of abelian ideals with given dimension.
The method is just to enumerate all cases of incomparable admissible
subsets $\Psi_{\min}$ by the Hasse diagram of $\Omega$ and the
associative restrictive conditions, and then to count the total
number of elements in $\Psi$. However, we will elide the enumerating
process here to shorten the length of this paper.

\subsection{Type $G_2$}\

For type $G_2$, we have
\begin{equation}
\Pi=\{\alpha_1,\alpha_2\}
\end{equation} and
\begin{equation}
\Phi_+=\{\alpha_1,\alpha_2,\alpha_1+\alpha_2,\alpha_1+2\alpha_2,\alpha_1+3\alpha_2,2\alpha_1+3\alpha_2\}.
\end{equation}
The highest root is
\begin{equation}
\theta=2\alpha_1+3\alpha_2.
\end{equation}

The Hasse diagram of $\Omega$ is as follows:

\vspace{0.1cm}
\begin{picture}(0,0)
\put(39,4){$\alpha_1+2\alpha_2$} \put(55,-5){$|$}
\put(39,-15){$\alpha_1+3\alpha_2$} \put(55,-24){$|$}
\put(35,-34){$2\alpha_1+3\alpha_2$}\put(110,-30){(Figure 4)}
\end{picture}
\vspace{1.5cm}

It is clear that the sum of arbitrary two roots in $\Omega$ is not
less than or equal to $\theta$. So there are $4=2^2$ abelian ideals
in a Borel subalgebra of the simple Lie algebra of type $G_2$:
$\emptyset$, $\mathfrak{g}_{2\alpha_1+3\alpha_2}$,
$\mathfrak{g}_{\alpha_1+3\alpha_2}\oplus\mathfrak{g}_{2\alpha_1+3\alpha_2}$
and
$\mathfrak{g}_{\alpha_1+2\alpha_2}\oplus\mathfrak{g}_{\alpha_1+3\alpha_2}\oplus\mathfrak{g}_{2\alpha_1+3\alpha_2}$.

We list the number of abelian ideals with given dimension in the
following table:
\begin{eqnarray}
\begin{tabular}{|c|c|c|c|c|c|}
\hline \mbox{dimension} &0&1&2&3&\mbox{total}\\
\hline \mbox{number} &1&1&1&1&4 \\
\hline
\end{tabular}
\end{eqnarray}

\subsection{Type $F_4$}\

For type $F_4$, we have
\begin{equation}
\Pi=\{\alpha_1=\epsilon_2-\epsilon_3,
\alpha_2=\epsilon_3-\epsilon_4,\alpha_3=\epsilon_4,\alpha_4=\frac{1}{2}(\epsilon_1-\epsilon_2-\epsilon_3-\epsilon_4)\}.
\end{equation} and
\begin{equation}
\Phi_+=\{\epsilon_1,\epsilon_2,\epsilon_3,\epsilon_4,\epsilon_1\pm\epsilon_2,
\epsilon_1\pm\epsilon_3,\epsilon_1\pm\epsilon_4,\epsilon_2\pm\epsilon_3,\epsilon_2\pm\epsilon_4,\epsilon_3\pm\epsilon_4,
\frac{1}{2}(\epsilon_1\pm\epsilon_2\pm\epsilon_3\pm\epsilon_4)\}.
\end{equation}
The highest root is
\begin{equation}
\theta=\epsilon_1+\epsilon_2=2\alpha_1+3\alpha_2+4\alpha_3+2\alpha_4.
\end{equation}

The Hasse diagram of $\Omega$ is as follows:

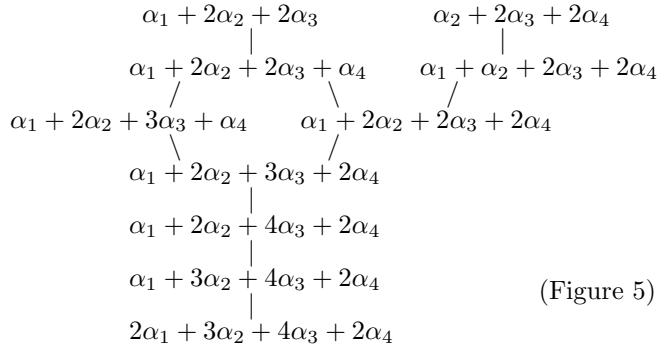
\begin{picture}(0,0)
\put(30,-5){$\alpha_1+2\alpha_2+2\alpha_3$}
\put(140,-5){$\alpha_2+2\alpha_3+2\alpha_4$} \put(70,-15){$|$}
\put(165,-15){$|$}
\put(25,-25){$\alpha_1+2\alpha_2+2\alpha_3+\alpha_4$}\put(135,-25){$\alpha_1+\alpha_2+2\alpha_3+2\alpha_4$}
\put(40,-35){$/$} \put(100,-35){$\backslash$}\put(145,-35){$/$}
\put(-20,-45){$\alpha_1+2\alpha_2+3\alpha_3+\alpha_4$}
\put(90,-45){$\alpha_1+2\alpha_2+2\alpha_3+2\alpha_4$}
\put(40,-55){$\backslash$} \put(100,-55){$/$}
\put(25,-65){$\alpha_1+2\alpha_2+3\alpha_3+2\alpha_4$}
\put(70,-75){$|$}
\put(25,-85){$\alpha_1+2\alpha_2+4\alpha_3+2\alpha_4$}
\put(70,-95){$|$}
\put(25,-105){$\alpha_1+3\alpha_2+4\alpha_3+2\alpha_4$}
\put(70,-115){$|$}
\put(25,-125){$2\alpha_1+3\alpha_2+4\alpha_3+2\alpha_4$}
\put(180,-110){(Figure 5)}
\end{picture}
\vspace{4.5cm}

when we take incomparable admissible subsets, there are some
restrictive conditions induced by
\begin{eqnarray}
(\alpha_1+2\alpha_2+2\alpha_3)+(\alpha_1+\alpha_2+2\alpha_3+2\alpha_4)=\theta.
\end{eqnarray}

By Figure 5 and the restrictive conditions, we can list the number
of abelian ideals with given dimension in the following table:

\begin{eqnarray}
\begin{tabular}{|c|c|c|c|c|c|c|c|c|c|c|c|}
\hline \mbox{dimension} &0&1&2&3&4&5&6&7&8&9&\mbox{total}\\
\hline \mbox{number} &1&1&1&1&1&2&2&3&3&1&16\\
\hline
\end{tabular}
\end{eqnarray}

\subsection{Type $E_6$}\

For type $E_6$, we have
\begin{equation}
\Pi=\{\alpha_1=\frac{1}{2}(\epsilon_8-\sum_{j=2}^7\epsilon_j+\epsilon_1),
\alpha_2=\epsilon_2-\epsilon_1,\alpha_3=\epsilon_3-\epsilon_2,\alpha_4=\epsilon_4-\epsilon_3,
\alpha_5=\epsilon_5-\epsilon_4,\alpha_6=\epsilon_1+\epsilon_2\}
\end{equation}
and
\begin{eqnarray}
&&\Phi_+=\{\epsilon_i\pm\epsilon_j,
\frac{1}{2}(\epsilon_8-\epsilon_7-\epsilon_6+\sum_{k=1}^5\pm\epsilon_k)\mid
1\leq j<i\leq5,\\\nonumber &&\mbox{ the total number of + signs is
even in
}\frac{1}{2}(\epsilon_8-\epsilon_7-\epsilon_6+\sum_{k=1}^5\pm\epsilon_k)\}.
\end{eqnarray}
The highest root is
\begin{equation}
\theta=\frac{1}{2}(\epsilon_1+\epsilon_2+\epsilon_3+\epsilon_4+\epsilon_5-\epsilon_6-\epsilon_7+\epsilon_8)
=\alpha_1+2\alpha_2+3\alpha_3+2\alpha_4+\alpha_5+2\alpha_6.
\end{equation}
The Hasse diagram of $\Omega$ is as follows:

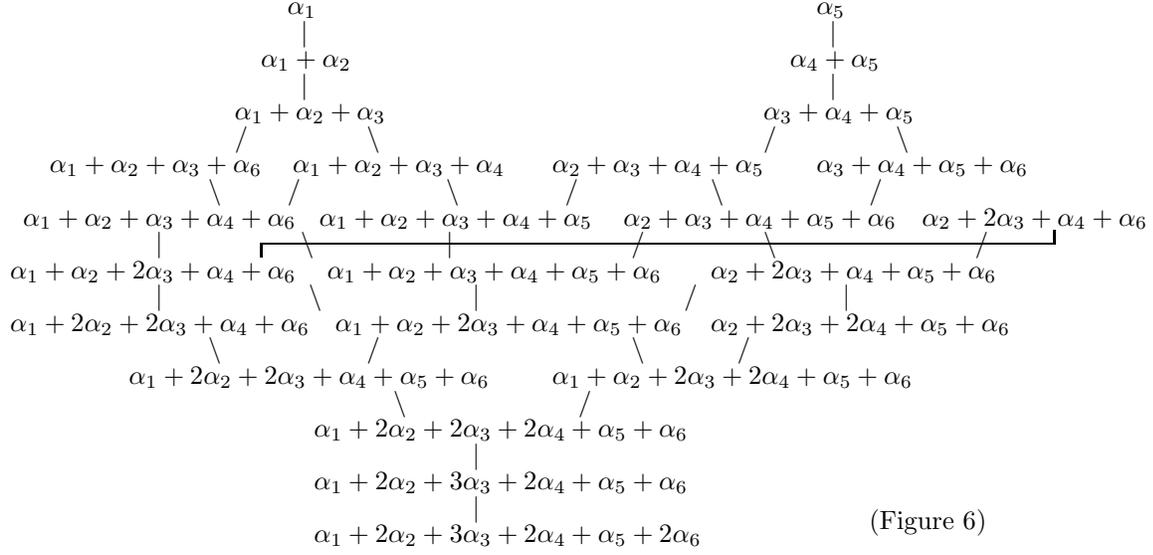
\begin{picture}(0,0)
\put(40,-5){$\alpha_1$} \put(240,-5){$\alpha_5$} \put(45,-15){$|$}
\put(245,-15){$|$} \put(30,-25){$\alpha_1+\alpha_2$}
\put(230,-25){$\alpha_4+\alpha_5$} \put(45,-35){$|$}
\put(245,-35){$|$} \put(20,-45){$\alpha_1+\alpha_2+\alpha_3$}
\put(220,-45){$\alpha_3+\alpha_4+\alpha_5$}
\put(20,-55){$/$}\put(70,-55){$\backslash$}\put(220,-55){$/$}\put(270,-55){$\backslash$}
\put(-50,-65){$\alpha_1+\alpha_2+\alpha_3+\alpha_6$}
\put(42,-65){$\alpha_1+\alpha_2+\alpha_3+\alpha_4$}
\put(140,-65){$\alpha_2+\alpha_3+\alpha_4+\alpha_5$}
\put(240,-65){$\alpha_3+\alpha_4+\alpha_5+\alpha_6$}
\put(10,-75){$\backslash$} \put(40,-75){$/$}
\put(100,-75){$\backslash$} \put(145,-75){$/$}
\put(200,-75){$\backslash$} \put(260,-75){$/$}
\put(-60,-85){$\alpha_1+\alpha_2+\alpha_3+\alpha_4+\alpha_6$}
\put(52,-85){$\alpha_1+\alpha_2+\alpha_3+\alpha_4+\alpha_5$}
\put(167,-85){$\alpha_2+\alpha_3+\alpha_4+\alpha_5+\alpha_6$}
\put(280,-85){$\alpha_2+2\alpha_3+\alpha_4+\alpha_6$}
\put(-10,-95){$|$} \put(45,-95){$\backslash$} \put(100,-95){$|$}
\put(170,-95){$/$} \put(220,-95){$\backslash$} \put(300,-95){$/$}
\put(30,-92){\line(6,0){300}} \put(30,-97){\line(0,1){5}}
\put(330,-92){\line(0,1){5}}
\put(-65,-105){$\alpha_1+\alpha_2+2\alpha_3+\alpha_4+\alpha_6$}
\put(55,-105){$\alpha_1+\alpha_2+\alpha_3+\alpha_4+\alpha_5+\alpha_6$}
\put(200,-105){$\alpha_2+2\alpha_3+\alpha_4+\alpha_5+\alpha_6$}
\put(-10,-115){$|$} \put(48,-115){$\backslash$} \put(110,-115){$|$}
\put(190,-115){$/$} \put(250,-115){$|$}
\put(-65,-125){$\alpha_1+2\alpha_2+2\alpha_3+\alpha_4+\alpha_6$}
\put(58,-125){$\alpha_1+\alpha_2+2\alpha_3+\alpha_4+\alpha_5+\alpha_6$}
\put(200,-125){$\alpha_2+2\alpha_3+2\alpha_4+\alpha_5+\alpha_6$}
\put(10,-135){$\backslash$} \put(70,-135){$/$}
\put(170,-135){$\backslash$} \put(210,-135){$/$}
\put(-20,-145){$\alpha_1+2\alpha_2+2\alpha_3+\alpha_4+\alpha_5+\alpha_6$}
\put(140,-145){$\alpha_1+\alpha_2+2\alpha_3+2\alpha_4+\alpha_5+\alpha_6$}
\put(80,-155){$\backslash$} \put(150,-155){$/$}
\put(50,-165){$\alpha_1+2\alpha_2+2\alpha_3+2\alpha_4+\alpha_5+\alpha_6$}
\put(110,-175){$|$}
\put(50,-185){$\alpha_1+2\alpha_2+3\alpha_3+2\alpha_4+\alpha_5+\alpha_6$}
\put(110,-195){$|$}
\put(50,-205){$\alpha_1+2\alpha_2+3\alpha_3+2\alpha_4+\alpha_5+2\alpha_6$}
\put(260,-200){(Figure 6)}
\end{picture}
\vspace{7.5cm}\\
At the same time, there are also several
restrictive conditions induced by
\begin{eqnarray}
&&(\alpha_1+\alpha_2+\alpha_3+\alpha_6)+(\alpha_2+2\alpha_3+2\alpha_4+\alpha_5+\alpha_6)\\
\nonumber&=&(\alpha_3+\alpha_4+\alpha_5+\alpha_6)+(\alpha_1+2\alpha_2+2\alpha_3+\alpha_4+\alpha_6)\\
\nonumber&=&(\alpha_1+\alpha_2+\alpha_3+\alpha_4+\alpha_6)+(\alpha_2+2\alpha_3+\alpha_4+\alpha_5+\alpha_6)\\
\nonumber&=&(\alpha_2+\alpha_3+\alpha_4+\alpha_5+\alpha_6)+(\alpha_1+\alpha_2+2\alpha_3+\alpha_4+\alpha_6)\\
\nonumber&=&(\alpha_2+2\alpha_3+\alpha_4+\alpha_6)+(\alpha_1+\alpha_2+\alpha_3+\alpha_4+\alpha_5+\alpha_6)=\theta
\end{eqnarray}

By Figure 6 and the restrictive conditions, we can list the number
of abelian ideals with given dimension in the following table:

\begin{equation}
\begin{tabular}{|c|c|c|c|c|c|c|c|c|c|c|c|c|c|c|c|c|c|c|}
\hline \mbox{dimension} &0&1&2&3&4&5&6&7&8&9&10&11&12&13&14&15&16&total\\
\hline \mbox{number} &1&1&1&1&2&3&3&4&6&7&8&10&7&4&2&2&2&64 \\
\hline
\end{tabular}
\end{equation}

\subsection{Type $E_7$}\

For type $E_7$, we have
\begin{eqnarray}
\Pi&=&\{\alpha_1=\frac{1}{2}(\epsilon_8-\epsilon_7-\epsilon_6-\epsilon_5-\epsilon_4-\epsilon_3-\epsilon_2+\epsilon_1),\\
&&\quad\alpha_2=\epsilon_2-\epsilon_1,\alpha_3=\epsilon_3-\epsilon_2,\alpha_4=\epsilon_4-\epsilon_3,\nonumber\\
&&\quad\alpha_5=\epsilon_5-\epsilon_4,\alpha_6=\epsilon_6-\epsilon_5,\alpha_7=\epsilon_1+\epsilon_2\}\nonumber
\end{eqnarray}
and
\begin{eqnarray}
\Phi_+=\{\epsilon_i\pm\epsilon_j,\epsilon_8-\epsilon_7,
\frac{1}{2}(\epsilon_8-\epsilon_7+\sum_{k=1}^6\pm\epsilon_k)\mid
1\leq j<i\leq6,\\\mbox{ the total number of + signs is even in
}\frac{1}{2}(\epsilon_8-\epsilon_7+\sum_{k=1}^6\pm\epsilon_k)\}.\nonumber
\end{eqnarray}
The highest root is
\begin{equation}
\theta=\epsilon_8-\epsilon_7=2\alpha_1+3\alpha_2+4\alpha_3+3\alpha_4+2\alpha_5+\alpha_6+2\alpha_7.
\end{equation}
We use $(i_1i_2\ldots i_7)$ to denote the root
$i_1\alpha_1+i_2\alpha_2+\cdots+i_7\alpha_7$ for short. The Hasse
diagram of $\Omega$ is drawn as follows.

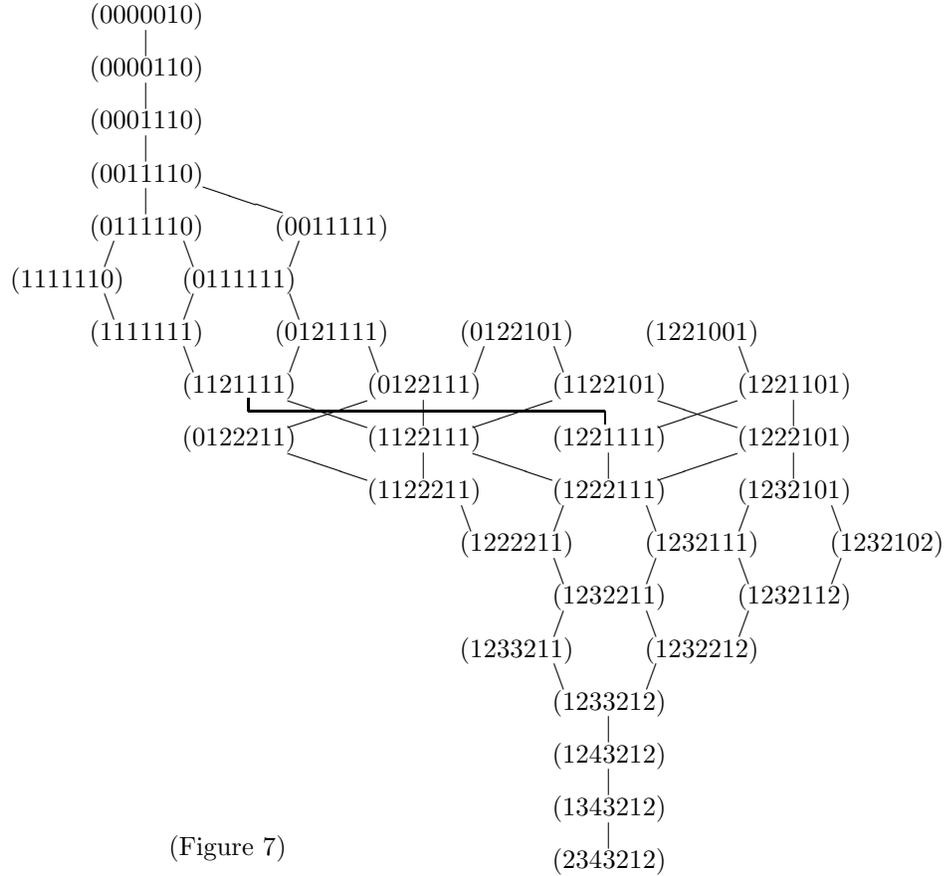
\begin{picture}(0,0)
\put(40,-5){(0000010)} \put(60,-15){$|$} \put(40,-25){(0000110)}
\put(60,-35){$|$} \put(40,-45){(0001110)} \put(60,-55){$|$}
\put(40,-65){(0011110)} \put(60,-75){$|$}
\put(83,-67){\line(3,-1){30}}

\put(40,-85){(0111110)} \put(110,-85){(0011111)} \put(45,-95){$/$}
\put(75,-95){$\backslash$} \put(115,-95){$/$}
\put(10,-105){(1111110)} \put(75,-105){(0111111)}
\put(45,-115){$\backslash$} \put(75,-115){$/$}
\put(115,-115){$\backslash$} \put(40,-125){(1111111)}
\put(110,-125){(0121111)} \put(180,-125){(0122101)}
\put(250,-125){(1221001)} \put(75,-135){$\backslash$}
\put(115,-135){$/$} \put(145,-135){$\backslash$} \put(185,-135){$/$}
\put(215,-135){$\backslash$} \put(285,-135){$\backslash$}
\put(75,-145){(1121111)} \put(145,-145){(0122111)}
\put(215,-145){(1122101)} \put(285,-145){(1221101)}
\put(115,-148){\line(3,-1){30}} \put(115,-158){\line(3,1){30}}
\put(165,-155){$|$} \put(185,-158){\line(3,1){30}}
\put(255,-158){\line(3,1){30}} \put(255,-148){\line(3,-1){30}}
\put(305,-155){$|$} \put(100,-152){\line(1,0){135}}
\put(100,-152){\line(0,1){5}} \put(235,-157){\line(0,1){5}}
\put(75,-165){(0122211)} \put(145,-165){(1122111)}
\put(215,-165){(1221111)} \put(285,-165){(1222101)}
\put(115,-168){\line(3,-1){30}} \put(165,-175){$|$}
\put(185,-168){\line(3,-1){30}} \put(235,-175){$|$}
\put(255,-178){\line(3,1){30}} \put(305,-175){$|$}
\put(145,-185){(1122211)} \put(215,-185){(1222111)}
\put(285,-185){(1232101)} \put(180,-195){$\backslash$}
\put(215,-195){$/$} \put(250,-195){$\backslash$} \put(285,-195){$/$}
\put(320,-195){$\backslash$} \put(180,-205){(1222211)}
\put(250,-205){(1232111)} \put(320,-205){(1232102)}
\put(215,-215){$\backslash$} \put(250,-215){$/$}
\put(285,-215){$\backslash$} \put(320,-215){$/$}
\put(215,-225){(1232211)} \put(285,-225){(1232112)}
\put(215,-235){$/$} \put(250,-235){$\backslash$} \put(285,-235){$/$}
\put(180,-245){(1233211)} \put(250,-245){(1232212)}
\put(215,-255){$\backslash$} \put(250,-255){$/$}
\put(215,-265){(1233212)} \put(235,-275){$|$}
\put(215,-285){(1243212)} \put(235,-295){$|$}
\put(215,-305){(1343212)} \put(235,-315){$|$}
\put(215,-325){(2343212)} \put(70,-320){(Figure 7)}
\end{picture}
\vspace{12cm}

The restrictive conditions are induced by
\begin{eqnarray}
&&(1111110)+(1232102)=(1111111)+(1232101)\\\nonumber&&=(1221001)+(1122211)
=(1121111)+(1222101)\\\nonumber&&=(1122101)+(1221111)=(1221101)+(1122111)=\theta.
\end{eqnarray}

By Figure 7 and the restrictive conditions, we can list the number
of abelian ideals with given dimension in the following table:

\begin{eqnarray}
\begin{tabular}{|c|c|c|c|c|c|c|c|c|c|c|c|c|c|c|}
\hline \mbox{dimension} &0&1&2&3&4&5&6&7&8&9&10&11&12&13\\
\hline \mbox{number} &1&1&1&1&1&2&2&3&3&4&5&6&7&8 \\
\hline
\end{tabular}\\
\begin{tabular}{|c|c|c|c|c|c|c|c|c|c|c|c|c|c|c|}
\hline 14&15&16&17&18&19&20&21&22&23&24&25&26&27&total\\
\hline  10&11&13&15&11&7&5&3&3&1&1&1&1&1& 128 \\
\hline
\end{tabular}\nonumber
\end{eqnarray}

\subsection{Type $E_8$}
\

For type $E_8$, we have
\begin{equation}
\Pi=\{\alpha_1=\frac{1}{2}(\epsilon_8-\sum_{k=2}^7\epsilon_k+\epsilon_1),
\alpha_2=\epsilon_2-\epsilon_1,\ldots,\alpha_7=\epsilon_7-\epsilon_6,\alpha_8=\epsilon_1+\epsilon_2\}
\end{equation} and
\begin{eqnarray}
\Phi_+&=&\{\epsilon_i\pm\epsilon_j,\frac{1}{2}(\epsilon_8+\sum_{k=1}^7\pm\epsilon_k)\mid
1\leq j<i\leq8,\\\nonumber&&\quad\mbox{ the total number of + signs
is even in }\frac{1}{2}(\epsilon_8+\sum_{k=1}^7\pm\epsilon_k)\}.
\end{eqnarray}
The highest root is
\begin{equation}
\theta=\epsilon_8+\epsilon_7
=2\alpha_1+4\alpha_2+6\alpha_3+5\alpha_4+4\alpha_5+3\alpha_6+2\alpha_7+3\alpha_8.
\end{equation}
We use $(i_1i_2\ldots i_8)$ to denote the root
$i_1\alpha_1+i_2\alpha_2+\cdots+i_8\alpha_8$ for short. The Hasse
diagram of $\Omega$ is drawn as follows.

\begin{picture}(0,0)
\put(40,-5){(01222211)} \put(200,-5){(12321002)} \put(60,-15){$|$}
\put(220,-15){$|$} \put(40,-25){(11222211)}
\put(200,-25){(12321102)} \put(60,-35){$|$} \put(200,-35){$/$}
\put(240,-35){$\backslash$} \put(40,-45){(12222211)}
\put(110,-45){(12332101)} \put(170,-45){(12321112)}
\put(230,-45){(12322102)} \put(60,-55){$|$} \put(130,-55){$|$}
\put(140,-52){\line(6,0){100}} \put(140,-52){\line(0,1){5}}
\put(240,-57){\line(0,1){5}} \put(190,-55){$|$}
\put(210,-58){\line(2,1){20}} \put(250,-55){$|$}
\put(40,-65){(12322211)} \put(110,-65){(12332111)}
\put(170,-65){(12322112)} \put(92,-68){\line(2,-1){20}}
\put(230,-65){(12332102)} \put(60,-75){$|$}
\put(92,-78){\line(2,1){20}} \put(155,-78){\line(2,1){20}}
\put(155,-68){\line(2,-1){20}} \put(190,-75){$|$}
\put(210,-78){\line(2,1){20}} \put(250,-75){$|$}
\put(40,-85){(12332211)} \put(110,-85){(12322212)}
\put(170,-85){(12332112)} \put(230,-85){(12432102)}
\put(60,-95){$|$} \put(130,-95){$|$} \put(90,-88){\line(2,-1){20}}
\put(155,-98){\line(2,1){20}} \put(190,-95){$|$} \put(250,-95){$|$}
\put(210,-98){\line(2,1){20}} \put(40,-105){(12333211)}
\put(110,-105){(12332212)} \put(170,-105){(12432112)}
\put(230,-105){(13432102)} \put(155,-118){\line(2,1){20}}
\put(92,-118){\line(2,1){20}} \put(210,-118){\line(2,1){20}}
\put(60,-115){$|$} \put(130,-115){$|$} \put(190,-115){$|$}
\put(250,-115){$|$} \put(40,-125){(12333212)}
\put(110,-125){(12432212)} \put(170,-125){(13432112)}
\put(230,-125){(23432102)} \put(80,-135){$\backslash$}
\put(110,-135){$/$} \put(150,-135){$\backslash$} \put(170,-135){$/$}
\put(200,-135){$\backslash$} \put(230,-135){$/$}
\put(70,-145){(12433212)} \put(130,-145){(13432212)}
\put(190,-145){(23432112)} \put(80,-155){$/$}
\put(110,-155){$\backslash$} \put(150,-155){$/$}
\put(170,-155){$\backslash$} \put (200,-155){$/$}
\put(40,-165){(12443212)} \put(110,-165){(13433212)}
\put(170,-165){(23432212)} \put(80,-175){$\backslash$}
\put(110,-175){$/$} \put(150,-175){$\backslash$} \put(170,-175){$/$}
\put(70,-185){(13443212)} \put(130,-185){(23433212)}
\put(80,-195){$/$} \put(110,-195){$\backslash$} \put(150,-195){$/$}
\put(40,-205){(13543212)} \put(110,-205){(23443212)}
\put(60,-215){$|$} \put(130,-215){$|$}
\put(90,-208){\line(2,-1){20}} \put(40,-225){(13543213)}
\put(110,-225){(23543212)} \put(80,-235){$\backslash$}
\put(110,-235){$/$} \put(140,-235){$\backslash$}
\put(70,-245){(23543213)} \put(130,-245){(24543212)}
\put(110,-255){$\backslash$} \put(140,-255){$/$}
\put(100,-265){(24543213)} \put(120,-275){$|$}
\put(100,-285){(24643213)} \put(120,-295){$|$}
\put(100,-305){(24653213)} \put(120,-315){$|$}
\put(100,-325){(24654213)} \put(120,-335){$|$}
\put(100,-345){(24654313)} \put(120,-355){$|$}
\put(100,-365){(24654323)} \put(230,-355){(Figure 8)}
\end{picture}

\newpage
The restrictive conditions are induced by
\begin{eqnarray}
&&(01222211)+(23432112)=(11222211)+(13432112)\\&=&(12222211)+(12432112)
=(12322211)+(12332112)\nonumber\\
&=&(12321112)+(12333211)=(12332111)+(12322212)\nonumber\\
&=&(12322112)+(12332211)=\theta\nonumber
\end{eqnarray}

By Figure 8 and the restrictive conditions, we can list the number
of abelian ideals with given dimension in the following table:

\begin{eqnarray}
\begin{tabular}{|c|c|c|c|c|c|c|c|c|c|c|c|c|c|}
\hline \mbox{dimension} &0&1&2&3&4&5&6&7&8&9&10&11&12\\
\hline \mbox{number} &1&1&1&1&1&1&1&2&2&2&2&3&3 \\
\hline
\end{tabular}\\
\begin{tabular}{|c|c|c|c|c|c|c|c|c|c|c|c|c|}
\hline 13&14&15&16&17&18&19&20&21&22&23&24&25\\
\hline  4&5&5&5&6&7&8&9&10&11&12&14&15 \\
\hline
\end{tabular}\nonumber\\
\begin{tabular}{|c|c|c|c|c|c|c|c|c|c|c|c|}
\hline 26&27&28&29&30&31&32&33&34&35&36&\mbox{total}\\
\hline  17&18&20&22&17&12&8&5&3&1&1&256\\
\hline
\end{tabular}\nonumber
\end{eqnarray}

\vspace{1cm}

\vspace{0.5cm} \noindent{\Large \bf Acknowledgements}

Thanks are due to my thesis advisor Prof. Xiaoping Xu, from whom I
learned so much.

\bibliographystyle{amsplain}

\end{document}